\documentclass{amsart}
\usepackage{amssymb}
\usepackage{hyperref}
\usepackage{cite}

 \newtheorem{thm}{Theorem}[section]
 \newtheorem{cor}[thm]{Corollary}
 \newtheorem{lem}[thm]{Lemma}
 \newtheorem{prop}[thm]{Proposition}
 \theoremstyle{definition}
 \newtheorem{dfn}[thm]{Definition}
 \theoremstyle{remark}
 \newtheorem{rem}[thm]{Remark}
 \newtheorem{ex}[thm]{Example}
 \numberwithin{equation}{section}

\begin{document}

\title[ common  spectral properties for bounded linear operators]
{A note on the common  spectral properties for bounded linear operators
}

\author[]{Hassane Zguitti}

\address{Hassane Zguitti: 
Department of Mathematics, Dhar El Mahraz Faculty of Science, Sidi Mohamed Ben Abdellah University, BO 1796 Fes-Atlas, 30003 Fez Morocco.
}

\email{hassane.zguitti@usmba.ac.ma}
\subjclass[2010]{47A10, 47A53, 47A55.}
\keywords{Jacobson's lemma, common spectral properties, regularity}

\begin{abstract} Let $X$ and $Y$ be Banach spaces, $A\,:\,X\rightarrow Y$ and $B,\,C\,:\,Y\rightarrow X$ be bounded linear operators. We prove that if $A(BA)^2=ABACA=ACABA=(AC)^2A,$ then 
$$\sigma_{*}(AC)\setminus\{0\}=\sigma_{*}(BA)\setminus\{0\}$$ where $\sigma_*$  runs over a large of spectra originated by regularities.

\end{abstract}

\maketitle

\section{Introduction}
Throughout this paper $\mathcal{L}(X,Y)$ denotes the set of all bounded linear operators acting from a complex Banach space $X$ into another one, $Y$,  and $\mathcal{L}(X)$ is a short for $\mathcal{L}(X,X)$. 
Given two  operators $A\in\mathcal{L}(X,Y)$ and $B\in\mathcal{L}(Y,X)$,   Jacobson's Lemma asserts that
\begin{equation}\label{Jacob} \sigma(AB)\setminus\{0\}=\sigma(BA)\setminus\{0\}
\end{equation}
where $\sigma(\cdot)$ denotes the ordinary spectrum.
 \smallskip
 
Several works have been devoted to equality (\ref{Jacob})  by showing that $AB-I$ and $BA-I$ share many spectral properties. See \cite{Bar, BZ1, CS, CDH, MZ, Sch1, Sch2, ZZ0, ZZ1} and the references therein. Barnes in \cite{Bar} extended (\ref{Jacob}) to other part of the spectrum and showed that $AB-I$ and $BA-I$ share some spectral properties. In \cite{BZ1}, Benhida and Zerouali investigated equation (\ref{Jacob}) for various Taylor joint spectra. For $A$ and $B$ satisfying $ABA=A^2$ and $BAB=B^2$,  Schmoeger \cite{Sch1, Sch2} and Duggal \cite{Du} showed that $A$, $B$, $AB$ and $BA$ share spectral properties. Corach {\it et al.} \cite{CDH} investigated common properties for $ac-1$ and $ba-1$ where $a, b$ and $c$ are elements in associative ring such that $aba=aca$. For bounded linear operators $A$, $B$ and $C$, Zeng and Zhong \cite{ZZ1} studied  spectral properties for $AC$ and $BA$ under the condition $ABA=ACA$. If $C=I$ in the last condition, one can retrieve Schmoeger's result. For operators $A$, $B$, $C$ and $D$ satisfying $ACD=DBD$ and $BDA=ACA$, Yan and Fang \cite{YF} investigated spectral properties for $AC$ and $BD$. Recently, \cite{CS} studied  common properties for $ac$ and $ba$ for elements in a ring satisfying $a(ba)^2=abaca=acaba=(ac)^2a$.
 \smallskip

 The paper is a continuation of \cite{CS} and \cite{Zg}. The aim of this paper is to extend recent results to bounded linear operators $A\in\mathcal{L}(X,Y)$ and $B,C\in\mathcal{L}(Y,X)$ satisfying
 $$A(BA)^2=ABACA=ACABA=(AC)^2A.$$
In section two we give basic definitions and notation which we need in the sequel. Section 3 is devoted to the main results of the paper. In Theorem \ref{thm.1} we prove that if $A\in\mathcal{L}(X,Y)$ and $B,C\in\mathcal{L}(Y,X)$ satisfy 
 $A(BA)^2=ABACA=ACABA=(AC)^2A,$
then 
$$\sigma_{*}(AC)\setminus\{0\}=\sigma_{*}(BA)\setminus\{0\}$$ where $\sigma_*$  runs over a large of spectra originated by regularities.
\section{Basic definitions and notations}
For an operator  $T\in\mathcal{L}(X)$, let $\mathcal{N}(T)$ and $\mathcal{R}(T)$ stand for the {\it kernel}, respectively the {\it range} of $T$. 
An operator $T\in\mathcal{L}(X)$ is said to be an {\it upper semi-Fredholm} operator if
 ${\mathcal R}(T)$ is closed and $\dim \mathcal{N}(T)<\infty$, and $T$ is said to be a {\it lower semi-Fredholm} operator if $\mbox{codim}\mathcal{N}(T)<\infty$. One says that $T$ is a  {\it Fredholm} operator if $\dim \mathcal{N}(T)<\infty$ and $\mbox{codim}\mathcal{N}(T)<\infty$. If $T$ is either upper or lower semi-Fredholm then $T$ is said {\it semi-Fredholm} operator.  In this case the {\it  index} of $T$ is defined by $\mbox{ind}(T)=\dim \mathcal{N}(T)-\dim \mathcal{R}(T)$.

The {\it ascent}  of $T$, $asc(T)$, is the smallest nonnegative integer  $n$ for which $N(T^n)=N(T^{n+1})$, i.e.; $asc(T)=\inf\{n\in\mathbb{Z}_+\,:\,\mathcal{N}(T^n)=\mathcal{N}(T^{n+1})\}$. If no such integer exists, we shall say that $T$ has infinite ascent. In a similar way,  the {\it descent} of $T$, $dsc(T)$ , is defined by $dsc(T)=\inf\{n\in\mathbb{Z}_+\,:\,\mathcal{R}(T^n)=\mathcal{R}(T^{n+1})\}$ and if no such integer exists, we shall say that $T$ has infinite descent. We say that $T$ is {\it left Drazin invertible} if $asc(T)<\infty$ and $\mathcal{R}(T^{asc(T)+1})$ is closed and $T$ is   {\it right Drazin invertible} if $dsc(T)<\infty$ and $\mathcal{R}(T^{dsc(T)})$ is closed. If $T$ is both left and right Drazin invertible, then $T$ is said to be {\it Drazin invertible} ; which is equivalent to  $asc(T)=dsc(T)<\infty$ (see  \cite{Ai_pams}). One says that $T$ is {\it upper semi-Browder} if $T$ is upper semi-Fredholm with finite ascent, and $T$ is {\it lower semi-Browder} if $T$ is lower semi-Fredholm with finite descent. If $T$ is both upper and lower semi-Browder then $T$ is said to be {\it Browder} operator (see \cite{Mu3}).
\smallskip

For each $n\in\mathbb{Z}_+$, let $c_n(T)=\dim R(T^n)/R(T^{n+1})$ and $c_n'(T)=\dim N(T^{n+1})/N(T^n).$
It was proved in \cite[Lemma 3.2]{Gr} that for every $n$, we have $$c_n(T)=\dim X/(R(T)+N(T^n))\mbox{ and }c_n'(T)=\dim N(T)\cap R(T^n).$$ It is easy to see that   $\{c_n(T)\}$ and $\{c_n'(T)\}$ are decreasing sequences and $dsc(T)=\inf\{n\in\mathbb{Z}_+\,:\,c_n(T)=0\},$ 
$asc(T)=\inf\{n\in\mathbb{Z}_+\,:\,c_n'(T)=0\}.$
\smallskip

Following \cite{Mu2},  the {\it essential descent} $dsc_e(T)$ of $T$ is defined by $dsc_e(T)=\inf\{n\in\mathbb{Z}_+\,:\,c_n(T)<\infty\},$ and the {\it essential ascent} $asc_e(T)$ of $T$ is defined by $asc_e(T)=\inf\{n\in\mathbb{Z}_+\,:\,c_n'(T)<\infty\},$ where the infimum over the empty set is taken to be infinite.
\smallskip

Let $\mathcal{N}^\infty(T)$ and $\mathcal{R}^\infty(T)$ denote the {\it hyper-kernel} and the {\it hyper-range} of $T$ defined by $$\mathcal{N}^\infty(T)=\displaystyle{\bigcup_{n=1}^\infty}\mathcal{N}(T^n)\mbox{  and }\mathcal{R}^\infty(T)=\displaystyle{\bigcap_{n=1}^\infty}\mathcal{R}(T^n).$$ One says that $T$ is  {\it semi-regular} if $\mathcal{R}(T)$ is closed and $\mathcal{N}^\infty(T)\subseteq\mathcal{R}(T)$.
\smallskip

For each $n\in\mathbb{Z}_+$, $T\in\mathcal{L}(X)$ induces a linear maps  $\Gamma_n$ from the space $\mathcal{R}(T^n)/\mathcal{R}(T^{n+1})$ into $ \mathcal{R}(T^{n+1})/\mathcal{R}(T^{n+2})$. The dimension of the null space of $\Gamma_n$ will be denoted by $k_n(T)$, i.e.,  $k_n(T)=\dim \mathcal{N}(\Gamma_n).$ 
It follows from \cite[Theorem 3.7]{Gr} that for every $n$, 
\[\begin{array}{lcl}k_n(T)&=&\dim((\mathcal{R}(T^n)\cap \mathcal{N}(T))/(\mathcal{R}(T^{n+1})\cap \mathcal{N}(T)))\\
&=&\dim(\mathcal{R}(T)+\mathcal{R}(T^{n+1}))/(\mathcal{R}(T)+\mathcal{N}(T^n)).
\end{array}
\]
 Let $$k(T)=\displaystyle{\sum_{n=0}^\infty}k_n(T).$$
 Then  it follows from \cite[Theorem 3.7]{Gr} that  
 $k(T)=\dim \mathcal{N}(T)/(\mathcal{N}(T)\cap \mathcal{R}^\infty(T))=\dim(\mathcal{R}(T)+\mathcal{N}^\infty(T))/\mathcal{R}(T).$ 
 The {\it stable nullity} $c(T)$ and the {\it stable defect} $c'(T)$ of $T$ are defined by 
 $$c(T)=\displaystyle{\sum_{n=0}^\infty}c_n(T)\mbox{ and }c'(T)=\displaystyle{\sum_{n=0}^\infty}c_n'(T).$$
Then we have $c(T)=\dim X/\mathcal{R}^\infty(T)\mbox{ and }c'(T)=\dim \mathcal{R}^\infty(T).$
 \smallskip
 
According to  \cite{Lab}, the {\it degree of stable iteration} of $T\in\mathcal{L}(X)$ is defined by $$dis(T)=\inf\{n\in\mathbb{Z}_+\,:\,k_m(T)=0\mbox{ for all }m\geq n\},$$ 
 and  the {\it degree of essential stable iteration} of $T$ (\cite{ZZ0}) is defined is $$dis_e(T)=\inf\{n\in\mathbb{Z}_+\,:\,k_m(T)<\infty\mbox{ for all }m\geq n\}.$$ 
\begin{dfn}
Let $R$ be a non-empty subset of ${\mathcal L}(X)$. $R$ is called a {\it regularity} if it satisfies the following two conditions:
\begin{itemize}
\item[i)] if $n\in\mathbb{N}$, then $A\in R$ if and only if $A^n\in R$ ;
\item[ii)] if $A,B,C$ and $D$ are mutually commuting operators in $\mathcal{L}(X)$ such that $AC+BD=I$, then $AB\in R\mbox{ if and only if }A\in R\mbox{ and }B\in R.$
\end{itemize}
\end{dfn}
 A regularity $R\subset{\mathcal L}(X)$ assigns to each $T\in{\mathcal L}(X)$ a subset of $\mathbb{C}$ defined by $$\sigma_{R}(T)=\{\lambda\in\mathbb{C}\,:\,T-\lambda I\notin R\}$$
and called the {\it spectrum of $T$ corresponding to the regularity} $R$. We note that every regularity $R$ contains all invertible operators, so that $\sigma_{R}(T)\subseteq \sigma(T)$. In general, $\sigma_{R}(T)$ is neither compact nor non-empty (see \cite{Mu1, Mu2, Mu3}).
\smallskip

The regularities $R_i$, where $\,1\leq i\leq 15$, were introduced and studied in \cite{Mu1, Mu2, Mu3} but are in a different form. Regularity $R_{18}$ was introduced by \cite{Ber}, while  $R_{16}, R_{17}$ and $R_{19}$ were introduced by \cite{ZZ0}.

\begin{dfn} 
\[\begin{array}{lcl}
R_1&=&\{T\in\mathcal{L}(X)\,:\,c(T)=0\},\\
R_2&=&\{T\in\mathcal{L}(X)\,:\,c(T)<\infty\},\\
R_3&=&\{T\in\mathcal{L}(X)\,:\mbox{ there exists }d\in\mathbb{Z}_+ \mbox{ such that }c_d(T)=0\mbox{ and }\mathcal{R}(T^{d+1})\mbox{ is closed}\},\\
R_4&=&\{T\in\mathcal{L}(X)\,:\,c_n(T)<\infty, \forall n\in\mathbb{Z}_+\},\\
R_5&=&\{T\in\mathcal{L}(X)\,:\,\mbox{ there exists } d\in\mathbb{Z}_+\mbox{ such that }c_d(T)<\infty\mbox{ and }\mathcal{R}(T^{d+1})\mbox{ is closed}\},\\
R_6&=&\{T\in\mathcal{L}(X)\,:\,c'(T)=0\mbox{ and }\mathcal{R}(T)\mbox{ is closed}\},\\
R_7&=&\{T\in\mathcal{L}(X)\,:\,c'(T)<\infty\mbox{ and }\mathcal{R}(T)\mbox{ is closed}\},\\
R_8&=&\{T\in\mathcal{L}(X)\,:\,\mbox{ there exists } d\in\mathbb{Z}_+\mbox{ such that }c_d'(T)=0\mbox{ and }\mathcal{R}(T^{d+1})\mbox{ is closed}\},\\
R_9&=&\{T\in\mathcal{L}(X)\,:\,c_n'(T)<\infty\mbox{ for every }n\in\mathbb{Z}_+\mbox{ and }\mathcal{R}(T)\mbox{ is closed}\},\\
R_{10}&=&\{T\in\mathcal{L}(X)\,:\,\mbox{ there exists } d\in\mathbb{Z}_+\mbox{ such that }c_d'(T)<\infty\mbox{ and }\mathcal{R}(T^{d+1})\mbox{ is closed}\},\\
R_{11}&=&\{T\in\mathcal{L}(X)\,:\,k(T)=0\mbox{ and }\mathcal{R}(T)\mbox{ is closed}\},\\
R_{12}&=&\{T\in\mathcal{L}(X)\,:\,k(T)<\infty\mbox{ and }\mathcal{R}(T)\mbox{ is closed}\},\\
R_{13}&=&\{T\in\mathcal{L}(X)\,:\,\mbox{ there exists } d\in\mathbb{Z}_+\mbox{ such that }k_n(T)=0\mbox{ for every }n\geq d\mbox{ and }\mathcal{R}(T^{d+1})\mbox{ is closed}\},\\
R_{14}&=&\{T\in\mathcal{L}(X)\,:\,k_n(T)<\infty\mbox{ for every }n\in\mathbb{Z}_+\mbox{ and }\mathcal{R}(T)\mbox{ is closed}\},\\
R_{15}&=&\{T\in\mathcal{L}(X)\,:\,\mbox{ there exists } d\in\mathbb{Z}_+\mbox{ such that }k_n(T)<\infty\mbox{ for every }n\geq d\mbox{ and }\mathcal{R}(T^{d+1})\mbox{ is closed}\},\\
R_{16}&=&\{T\in\mathcal{L}(X)\,:\mbox{ there exists }d\in\mathbb{Z}_+ \mbox{ such that }c_d(T)=0\mbox{ and }\mathcal{R}(T)+N(T^d)\mbox{ is closed}\},\\
R_{17}&=&\{T\in\mathcal{L}(X)\,:\mbox{ there exists }d\in\mathbb{Z}_+ \mbox{ such that }c_d(T)<\infty\mbox{ and }\mathcal{R}(T)+N(T^d)\mbox{ is closed}\},\\
R_{18}&=&\{T\in\mathcal{L}(X)\,:\,\exists d\in\mathbb{Z}_+\mbox{ such that }k_n(T)=0\mbox{ for every }n\geq d\mbox{ and }\mathcal{R}(T)+N(T^d)\mbox{ is closed}\},\\
R_{19}&=&\{T\in\mathcal{L}(X)\,:\,\exists d\in\mathbb{Z}_+\mbox{ such that }k_n(T)<\infty\mbox{ for every }n\geq d\mbox{ and }\mathcal{R}(T)+N(T^d)\mbox{ is closed}\}.
\end{array}
\]
\end{dfn} 

 We have $$R_{1}\subseteq R_{2}=R_{3}\cap R_{4}\subseteq R_{3}\cup R_{4}\subseteq R_{5}\subseteq R_{13},$$
 $$R_{6}\subseteq R_{7}=R_{8}\cap R_{9}\subseteq R_{8}\cup R_{9}\subseteq R_{10}\subseteq R_{13},$$
 $$R_{11}\subseteq R_{12}=R_{13}\cap R_{14}\subseteq R_{13}\cup R_{14}\subseteq R_{15}.$$
 
 It was proved in \cite[Proposition 2.7]{ZZ0} that 
 \[\begin{array}{lcl}
R_3&=&\{T\in\mathcal{L}(X)\,:\,dsc(T)<\infty\mbox{ and }\mathcal{R}(T^{dsc(T)+1})\mbox{ is closed}\},\\
R_5&=&\{T\in\mathcal{L}(X)\,:\,dsc_e(T)<\infty\mbox{ and }\mathcal{R}(T^{dsc_e(T)+1})\mbox{ is closed}\},\\
R_8&=&\{T\in\mathcal{L}(X)\,:\,asc(T)<\infty\mbox{ and }\mathcal{R}(T^{asc(T)+1})\mbox{ is closed}\},\\
R_{10}&=&\{T\in\mathcal{L}(X)\,:\,asc_e(T)<\infty\mbox{ and }\mathcal{R}(T^{asc_e(T)+1})\mbox{ is closed}\},\\
R_{13}&=&\{T\in\mathcal{L}(X)\,:\,dis(T)<\infty\mbox{ and }\mathcal{R}(T^{dis(T)+1})\mbox{ is closed}\},\\
R_{15}&=&\{T\in\mathcal{L}(X)\,:\,dis_e(T)<\infty\mbox{ and }\mathcal{R}(T^{dis_e(T)+1})\mbox{ is closed}\}.
\end{array}
\]
 
The operators of $R_{1},  R_{2}, R_{3},  R_{4}$ and $R_{5}$ are surjective, lower semi-Browder, right Drazin invertible, lower semi-Fredholm and right essentially Drazin invertible operators, respectively. The operators of $R_{6}, R_{7}, R_{8}, R_{9}$ and $R_{10}$   are bounded below, upper semi-Browder, left Drazin invertible, upper semi-Fredholm and left essentially Drazin invertible operators, respectively.  The operators of $R_{11}, R_{12}$ and $R_{13}$ are semi-regular, essentially semi-regular and quasi-Fredholm operators. The operators of $R_{18}$ are the operators with eventual topological uniform descent.     
 \smallskip

\section{Main results}
The following is our main result.
\begin{thm}\label{thm.1} Let $A\in\mathcal{L}(X,Y)$ and $B,C\in\mathcal{L}(Y,X)$ such that $A(BA)^2=ABACA=ACABA=(AC)^2A.$
Then $$\sigma_{R_i}(AC)\setminus\{0\}=\sigma_{R_i}(BA)\setminus\{0\}\mbox{ for } 1\leq i\leq 19.$$
\end{thm}
The proof of our main result uses several auxiliary lemmas.
\begin{lem}\label{l1.1} Let $A\in\mathcal{L}(X,Y)$ and $B,C\in\mathcal{L}(Y,X)$ such that $A(BA)^2=ABACA=ACABA=(AC)^2A.$ Let $Q$ be a polynomial. Then we have\\
\indent 1) $ABA\mathcal{R}(Q(CA-I))\subseteq\mathcal{R}(Q(AB-I))$;\\
\indent 2) $ABA\mathcal{N}(Q(CA-I)\subseteq\mathcal{N}(Q(AB-I))$;\\
\indent 3) $ACA\mathcal{R}(Q(BA-I))\subseteq\mathcal{R}(Q(AC-I))$;\\
\indent 4) $ACA\mathcal{N}(Q(BA-I))\subseteq\mathcal{N}(Q(AC-I))$.
\end{lem}
\begin{proof} It is easy to see that for each $k\in\mathbb{Z}_+$, 
\begin{equation}\label{eqQ1}ABA(CA-I)^k=(AB-I)^kABA\mbox{ and }ACA(BA-I)^k=(AC-I)^kACA.
\end{equation}
Then \begin{equation}\label{eqQ2}ABAQ(CA-I)=Q(AB-I)ABA\mbox{ and }ACAQ(BA-I)=Q(AC-I)ACA.
\end{equation}
1) Let $x$ belongs to $\mathcal{R}(Q(CA-I))$. Then there exists some $y\in X$ such that $Q(CA-I)y=x$. Hence it follows from (\ref{eqQ1}) that  $ABAx=ABAQ(CA-I)x=Q(AB-I)ABAx$ which belongs to $\mathcal{R}(Q(AB-I))$. Thus $ABA\mathcal{R}(Q(CA-I))\subseteq\mathcal{R}(Q(AB-I))$.
\smallskip

2) Let $x\in\mathcal{N}(Q(CA-I))$. Then $Q(CA-I)x=0$. It follows from (\ref{eqQ1}) that  $Q(AB-I)ABAx=ABAQ(CA-I)x=0$. Thus $ABAx\in\mathcal{N}(Q(AB-I))$.
\smallskip

Using (\ref{eqQ2}), 3) and 4) go similarly.
\end{proof}

\begin{lem}\label{l1.4} Let $A\in\mathcal{L}(X,Y)$ and $B,C\in\mathcal{L}(Y,X)$ such that $A(BA)^2=ABACA=ACABA=(AC)^2A.$
Then $$c_n(AC-I)=c_n(BA-I)\mbox{ for all }n\in\mathbb{Z}_+.$$
In particular, $c(AC-I)=c(BA-I).$
\end{lem}
\begin{proof} Let $$\Gamma_{ACA}\,:\,\mathcal{R}((BA-I)^{n})/\mathcal{R}((BA-I)^{n+1})\rightarrow \mathcal{R}((AC-I)^{n})/\mathcal{R}((AC-I)^{n+1})$$ be the linear application defined by $$\Gamma_{ACA}(x+\mathcal{R}((BA-I)^{n+1}))=ACAx+\mathcal{R}((AC-I)^{n+1}).$$ Since $ACA\mathcal{R}((BA-I)^n)\subseteq\mathcal{R}((AC-I)^n)$ by Lemma \ref{l1.1}, part 3), then $\Gamma_{ACA}$ is well defined. We shall show that $\Gamma_{ACA}$ is injective.
\smallskip

Let $x\in\mathcal{R}((BA-I)^{n})$ such that $\Gamma_{ACA}(x)=0$. Then $ACAx\in\mathcal{R}((AC-I)^{n+1})$. Hence $CACAx\in \mathcal{R}((CA-I)^{n+1})$. From Lemma \ref{l1.1}, part 1), we have $ABACACAx\in \mathcal{R}((AB-I)^{n+1}).$ Then 
 $$(BA)^4x=BABACACAx\in \mathcal{R}((BA-I)^{n+1}).$$ 
Since $x\in\mathcal{R}((BA-I)^{n})$ then $x=(BA-I)^nz$ for some $z\in X$.
Hence
\[\begin{array}{lcl}x&=&(BA)^4x-((BA)^4-I)x\\
&=&(BA)^4x-((BA)^3+(BA)^2+(BA)+I)(BA-I)x\\
&=&(BA)^4x-((BA)^3+(BA)^2+(BA)+I)(BA-I)^{n+1}z\\
&=&(BA)^4x-(BA-I)^{n+1}(((BA)^3+(BA)^2+(BA)+I)z)\in\mathcal{R}((BA-I)^{n+1}).
\end{array}
\]
Thus $\Gamma_{ACA}$ is injective and consequently
\begin{equation}\label{eq1c} c_n(BA-I)\leq c_n(AC-I).
\end{equation}
In similar way, we show that  
\begin{equation}\label{eq2c} c_n(CA-I)\leq c_n(AB-I).
\end{equation}
Finally,
\[\begin{array}{lcl}c_n(BA-I)&\leq&c_n(AC-I)\\
&=&c_n(CA-I) \mbox{ (\cite[Lemma 3.9]{ZZ0}}\\
&\leq&c_n(AB-I) \mbox{ by (\ref{eq2c})}\\
&=&c_n(BA-I) \mbox{ (\cite[Lemma 3.9]{ZZ0}}.
\end{array}
\]
Therefore $c_n(BA-I)=c_n(AC-I)$ for all $n\in\mathbb{Z}_+$. In particular, $c(AC-I)=c(BA-I).$
\end{proof}

For $T\in\mathcal{L}(X)$, let $\sigma_{asc}(T)$ and $\sigma_{asc}^e(T)$ be, respectively, the {\it ascent spectrum} and the {\it  descent spectrum} of $T$ defined by $$\sigma_{acs}(T)=\{\lambda\in\mathbb{C}\,:\,asc(T)=\infty\}\mbox{ and }\sigma_{acs}^e(T)=\{\lambda\in\mathbb{C}\,:\,asc(T).$$
The following is an immediate consequence of Lemma \ref{l1.4}.
\begin{cor}\label{cor_asc} Let $A\in\mathcal{L}(X,Y)$ and $B,C\in\mathcal{L}(Y,X)$ such that $A(BA)^2=ABACA=ACABA=(AC)^2A.$ Then
$$\sigma_*{AC}\setminus\{0\}=\sigma_*{BA}\setminus\{0\},\mbox{ for }\sigma_*\in\{\sigma_{asc},\sigma_{acs}^e\}.$$
\end{cor}

\begin{lem}\label{l1.3} Let $A\in\mathcal{L}(X,Y)$ and $B,C\in\mathcal{L}(Y,X)$ such that $A(BA)^2=ABACA=ACABA=(AC)^2A.$ 
Then $$c_n'(AC-I)=c_n'(BA-I)\mbox{ for all }n\in\mathbb{Z}_+.$$
In particular, $c'(AC-I)=c'(BA-I).$
\end{lem}
\begin{proof} Let $$\Psi_{ACA}\,:\,\mathcal{N}((BA-I)^{n+1})/\mathcal{N}((BA-I)^{n})\rightarrow \mathcal{N}((AC-I)^{n+1})/\mathcal{N}((AC-I)^{n})$$ be the linear application defined by $$\Psi_{ACA}(x+\mathcal{N}((BA-I)^{n}))=ACAx+\mathcal{N}((AC-I)^{n}).$$ Since $ACA\mathcal{N}((BA-I)^{n+1})\subseteq\mathcal{N}((AC-I)^{n+1})$ by Lemma \ref{l1.1}, part 4), then $\Psi_{ACA}$ is well defined. 
\smallskip

Now we show that $\Psi_{ACA}$ is injective. Let $x\in\mathcal{N}((BA-I)^{n+1})$ such that $\Psi_{ACA}(x)=0$, which means that $ACAx\in\mathcal{N}((AC-I)^{n})$. Hence $CACAx\in \mathcal{N}((CA-I)^n)$. It follows from Lemma \ref{l1.1}, part ii), that $ABACACAx\in \mathcal{N}((AB-I)^n).$ Then $$(BA)^4x=BABACACAx\in \mathcal{N}((BA-I)^n).$$  Hence
\[\begin{array}{lcl}x&=&(BA)^4x-((BA)^4-I)x\\
&=&(BA)^4x-[(BA)^3+(BA)^2+(BA)+I](BA-I)x\in\mathcal{N}((BA-I)^n).
\end{array}
\]
Which implies that $\Psi_{ACA}$ is injective and then 
\begin{equation}\label{eq1c'} c_n'(BA-I)\leq c_n'(AC-I).
\end{equation}
Similarly,  we prove that
\begin{equation}\label{eq2c'} c_n'(CA-I)\leq c_n'(AB-I).
\end{equation}
Finally,
\[
\begin{array}{lcl}c_n'(BA-I)&\leq&c_n'(AC-I)\\
&=&c_n'(CA-I) \mbox{ (\cite[Lemma 3.10]{ZZ0}}\\
&\leq&c_n'(AB-I) \mbox{ by (\ref{eq2c'})}\\
&=&c_n'(BA-I) \mbox{ (\cite[Lemma 3.10]{ZZ0}} ;
\end{array}
\]
Therefore $c_n'(BA-I)=c_n'(AC-I)$ for all $n\in\mathbb{Z}_+$. In particular, $c'(AC-I)=c'(BA-I).$
\end{proof}

For $T\in\mathcal{L}(X)$ let $\sigma_{dsc}(T)$ and $\sigma_{dsc}^e(T)$ be respectively the {\it descent spectrum} and the {\it  essential descent spectrum} of $T$ defined by $$\sigma_{dcs}(T)=\{\lambda\in\mathbb{C}\,:\,dsc(T)=\infty\}\mbox{ and }\sigma_{dcs}^e(T)=\{\lambda\in\mathbb{C}\,:\,dsc(T).$$
Then the following is an  immediate consequence of Lemma \ref{l1.3} 
\begin{cor}\label{cor_asc} Let $A\in\mathcal{L}(X,Y)$ and $B,C\in\mathcal{L}(Y,X)$ such that $A(BA)^2=ABACA=ACABA=(AC)^2A.$ Then
$$\sigma_*{AC}\setminus\{0\}=\sigma_*{BA}\setminus\{0\},\mbox{ for }\sigma_*\in\{\sigma_{dsc},\sigma_{dcs}^e\}.$$
\end{cor}

\begin{lem}\label{l1.5} Let $A\in\mathcal{L}(X,Y)$ and $B,C\in\mathcal{L}(Y,X)$ such that $A(BA)^2=ABACA=ACABA=(AC)^2A.$
Then $$k_n(AC-I)=k_n(BA-I)\mbox{ for all }n\in\mathbb{Z}_+.$$
In particular, $k(AC-I)=k(BA-I).$
\end{lem}
\begin{proof} Let $\Phi_{ACA}$ be the linear application from $\mathcal{R}(BA-I)+\mathcal{N}((BA-I)^{n+1})/\mathcal{R}(BA-I)+\mathcal{N}((BA-I)^{n})$ to $ \mathcal{R}(AC-I)+\mathcal{N}((AC-I)^{n+1})/\mathcal{R}(AC-I)+\mathcal{N}((AC-I)^{n})$  defined by $$\Phi_{ACA}(x+\mathcal{R}(BA-I)+\mathcal{N}((BA-I)^{n}))=ACAx+\mathcal{R}(BA-I)+\mathcal{N}((AC-I)^{n}).$$ Since, by Lemme \ref{l1.1}, parts 3) and 4), 
$$ACA(\mathcal{R}(BA-I))+\mathcal{N}((BA-I)^{n+1})\subseteq\mathcal{R}(BA-I))+\mathcal{N}((BA-I)^{n+1}),$$
then $\Phi_{ACA}$ is well defined. 
\smallskip

We prove that $\Phi_{ACA}$ is injective. Let $x\in \mathcal{R}(BA-I)+\mathcal{N}((BA-I)^{n+1})$ such that $\Phi_{ACA}(x)=0$. Then $ACAx\in \mathcal{R}(AC-I)+\mathcal{N}((AC-I)^{n})$. So, there exist some $y\in \mathcal{R}(BA-I)$ and $z\in \mathcal{N}((AC-I)^{n})$ such that $ACAx=y+z$. Then $CACAx=Cy+Cz\in \mathcal{R}(CA-I)+\mathcal{N}((CA-I)^{n})$. Thus by Lemma \ref{l1.1}, parts 1) and 2), we get that $ABACACAx\in \mathcal{R}(AB-I)+\mathcal{N}((AB-I)^{n})$ and consequently $(BA)^4x=BABACACAx\in\mathcal{R}(BA-I)+\mathcal{N}((BA-I)^{n})$. Thus 
\[\begin{array}{lcl}x&=&(BA)^4x-((BA)^4-I)x\\
&=&(BA)^4x-(BA-I)((BA)^3+(BA)^2+(BA)+I)x\in\mathcal{R}(BA-I)+\mathcal{N}((BA-I)^n).
\end{array}
\]
Hence $\Phi_{ACA}$ is injective. Thus 
\begin{equation}\label{eq1k} k_n(BA-I)\leq k_n(AC-I).
\end{equation}
In similar way, we show that  
\begin{equation}\label{eq2k} k_n(CA-I)\leq k_n(AB-I).
\end{equation}
Therefore,
\[\begin{array}{lcl}k_n(BA-I)&\leq&k_n(AC-I)\\
&=&k_n(CA-I) \mbox{ (\cite[Lemma 3.8]{ZZ0}}\\
&\leq&k_n(AB-I) \mbox{ by (\ref{eq2k})}\\
&=&k_n(BA-I) \mbox{ (\cite[Lemma 3.8]{ZZ0}}.
\end{array}
\]

\end{proof}
\begin{lem}\label{l1.6} Let $A\in\mathcal{L}(X,Y)$ and $B,C\in\mathcal{L}(Y,X)$ such that $A(BA)^2=ABACA=ACABA=(AC)^2A.$
Then for all $n\in\mathbb{Z}_+$, $\mathcal{R}((AC-I)+\mathcal{N}((AC-I)^{n})$ is closed if and only if $\mathcal{R}(BA-I)+\mathcal{N}((BA-I)^{n})$ is closed.

In particular $\mathcal{R}(AC-I)$ is closed if and only if $\mathcal{R}(BA-I)$ is closed.
\end{lem}

\begin{proof} Assume that $\mathcal{R}(AC-I)+\mathcal{N}((AC-I)^{n})$ is closed. Let $\{x_p\}$ be a sequence in 
$\mathcal{R}(BA-I)+\mathcal{N}((BA-I)^{n})$ which converges to $x\in X$. Then $ACAx_p$ converge to $ACAx$. Since $ACA(\mathcal{R}(BA-I)+\mathcal{N}((BA-I)^{n}))\subset\mathcal{R}(AC-I)+\mathcal{N}((AC-I)^{n})$ by Lemma \ref{l1.1}, part 3) and 4), then $ACAx_p$ belongs to $\mathcal{R}((AC-I)+\mathcal{N}((AC-I)^{n})$. Since $\mathcal{R}(AC-I)+\mathcal{N}((AC-I)^{n})$ is closed and $ACAx_p$ converges to $ACAx$. 
\[\begin{array}{lcl}
\quad &\Longrightarrow&ACAx\in \mathcal{R}(AC-I)+\mathcal{N}((AC-I)^{n})\\
&\Longrightarrow&CACAx\in \mathcal{R}(CA-I)+\mathcal{N}((CA-I)^{n})\mbox{ }\\
&\Longrightarrow&ABA(CACAx)\in \mathcal{R}(AB-I)+\mathcal{N}((AB-I)^{n})\quad\mbox{ (by Lemma \ref{l1.1})}\\
&\Longrightarrow&(BA)^4x=ABA(CACAx)\in \mathcal{R}(AB-I)+\mathcal{N}((AB-I)^{n}).\
\end{array}
\]
Thus 
\[\begin{array}{lcl}x&=&(BA)^4x-((BA)^4-I)x\\
&=&(BA)^4x-(BA-I)((BA)^3+(BA)^2+(BA)+I)x\in\mathcal{R}(BA-I)+\mathcal{N}((BA-I)^n).
\end{array}
\]
Therefore $\mathcal{R}(BA-I))+\mathcal{N}((BA-I)^n)$ is closed.
\smallskip

The opposite implication goes similarly.
\end{proof}

\begin{lem}\label{l1.7} Let $A\in\mathcal{L}(X,Y)$ and $B,C\in\mathcal{L}(Y,X)$ such that $A(BA)^2=ABACA=ACABA=(AC)^2A.$
Then for all $n\in\mathbb{N}$, $\mathcal{R}((AC-I)^n)$ is closed if and only if $\mathcal{R}(BA-I)^n)$ is closed.
\end{lem}
\begin{proof} As in the presentation before \cite[Proposition]{Bar}, for each $n\in\mathbb{N}$ there exists $B_n$ and $C_n\in\mathcal{L}(Y,X)$ such that $$(I-AC)^n=I-AC_n\mbox{ and }(I-BA)^n=I-B_nA.$$
Indeed, we have $B_n=\displaystyle{\sum_{k=1}^n}(-1)^{k-1}({}_k^n)B(AB)^{k-1}$ and $C_n=\displaystyle{\sum_{k=1}^n}(-1)^{k-1}({}_k^n)(CA)^{k-1}C$. It is easy to check that $$A(B_nA)^2=AB_nAC_nA=AC_nAB_nA=(AC_n)^2A.$$ Then it follows from Lemma \ref{l1.6} that $\mathcal{R}((AC-I)^n)$ is closed if and only if $\mathcal{R}((BA-I)^n)$ is closed.
\end{proof}

\noindent {\bf {\it Proof of Theorem \ref{thm.1}}} : The proof follows at once from Lemmas \ref{l1.1}-\ref{l1.7}.
\section{Applications and concluding remarks}
A bounded operator $T\in\mathcal{L}(X)$ is said to be {\it upper semi-Weyl} operator if $T$ is upper semi-Fredholm with $ind(T)\leq 0$, and $T$ is said to be {\it lower semi-Weyl} operator if $T$ is lower semi-Fredholm with $ind(T)\geq 0$. If $T$ is both upper and lower semi-Fredholm then $T$ is said to {\it Weyl} operator. Then $T$ is weyl operator precisely when $T$ is a Fredholm operator with index zero. The {\it  upper semi-Weyl spectrum}  $\sigma_{uw}(T)$, the {\it lower semi-Weyl spectrum }  $\sigma_{lw}(T)$ and the {\it Weyl spectrum} $\sigma_w(T)$ of $T$ are defined by 
$$\sigma_{uw}(T)=\{\lambda\in\mathbb{C}\,:\,T-\lambda I\mbox{ is not upper semi-Weyl}\},$$
$$\sigma_{lw}(T)=\{\lambda\in\mathbb{C}\,:\,T-\lambda I\mbox{ is not lower semi-Weyl}\},$$
$$\sigma_w(T)=\sigma_{uw}(T)\cup\sigma_{lw}(T).$$
From Lemma \ref{l1.4} and Lemma \ref{l1.3} we deduce the following result 
\begin{prop} Let $A\in\mathcal{L}(X,Y)$ and $B,C\in\mathcal{L}(Y,X)$ such that $A(BA)^2=ABACA=ACABA=(AC)^2A.$ Then $$\sigma_*(AC)\setminus\{0\}=\sigma_*(BA)\setminus\{0\}\,\mbox{ for }\sigma_*\in\{\sigma_{uw},\sigma_{lw},\sigma_w\}.$$
\end{prop}

An operator $T\in\mathcal{L}(X)$ is said to be {\it Riesz} operator if $T-\lambda I$ is a Fredholm operator for all $0\neq \lambda\in\mathbb{C}$. Then the following proposition is an immediate consequence of Theorem \ref{thm.1}

\begin{prop}\label{prop3.1} Let $A\in\mathcal{L}(X,Y)$ and $B,C\in\mathcal{L}(Y,X)$ such that $A(BA)^2=ABACA=ACABA=(AC)^2A.$ Then $AC$ is a Riesz operator if and only if $BA$ is a Riesz operator.
\end{prop}

Following \cite{SZZ}, an operator $T\in\mathcal{L}(X)$ is said to be {\it generalized Drazin-Riesz} operator if there exists $S\in\mathcal{L}(X)$ such that 
$$TS=ST,\,STS=S\mbox{ and }T^2S-T\mbox{ is a Riesz operator}.$$
The operator $S$ is called a {\it generalized Drazin-Riesz inverse} of $T$.
\begin{thm} Let $A\in\mathcal{L}(X,Y)$ and $B,C\in\mathcal{L}(Y,X)$ such that $A(BA)^2=ABACA=ACABA=(AC)^2A.$ Then $AC$ is generalized Drazin-Riesz invertible if and only if $BA$ is generalized Drazin-Riesz invertible. In this case, if $S$ is a generalized Drazin-Riesz inverse of $AC$ then $BS^2A$ is a generalized Drazin-inverse of $BA$.
\end{thm}
\begin{proof}
Assume that $AC$ is generalized Drazin-Riesz invertible. then there exists $S\in\mathcal{L}(X)$ such that $S(AC)=(AC)S$, $S(AC)S=S$ and $(AC)^2S-AC$ is Riesz. Set $T=BS^2A$ and we shall show that  $$T(BA)=(BA)T, \,T(BA)T=T\mbox{ and }(BA)^2T-BA \mbox{ is Riesz operator}.$$
For the first equality, we have \[\begin{array}{lcl}T(BA)&=&BS^2A(BA)\\
&=&BS^2(AC)S^2(AC)A(BA)\\
&=&BS^4(AC)^2A(CA)\\
&=&B(AC)^3S^4A\\
&=&B(AB)S^2A\\
&=&BAT.
\end{array}
\]

For the second, 
 \[\begin{array}{lcl}T^2(BA)&=&BS^2ABS^2ABA\\
&=&BS^2ABS^2(AC)S^2(AC)ABA\\
&=&BS^2ABS^2(AC)S^2(AC)ACA\\
&=&BS^2AB(AC)(AC)S^4ACA\\
&=&BS^2AC(AC)(AC)S^4ACA\\
&=&BS^2ACS^2ACA\\
&=&BS^2A\\
&=&T.
\end{array}
\]
Set $P=ACS-I=SAC-I$. Then 
\[\begin{array}{lcl}T(BA)^2-BA&=&BS^2A(BA)^2-BA\\
&=&BS^2(AC)^2A-BA\\
&=&BS^ACA-BA\\
&=&B(SAC-I)A\\
&=&BPA.
\end{array}
\]
Hence it remains to show that $BPA$ is a Riesz operator. We have
\[\begin{array}{lcl}(PA)B(PA)B(PA)&=&(SACA-A)B(SACA-A)B(ACSA-A)\\
&=&(SACA-A)B(SACABA-ABA)(CSA-A)\\
&=&(SACA-A)B(SACACA-ABA)(CSA-A)\\
&=&[(SACA-A)B(SACACA)-(SACA-A)BABA](CSA-A)\\
&=&[(SACA-A)B(SACACA)-(SACA-A)BACA](CSA-A)\\
&=&(SACA-A)B(SACACA-ACA)(CSA-A)\\
&=&(SACA-A)B(SACA-A)C(ACSA-A)\\
&=&(PA)B(PA)C(PA).
\end{array}
\]
In the same way, one can prove that 
$$(PA)B(PA)B(PA)=(PA)B(PA)C(PA)=(PA)C(PA)B(PA)=(PA)C(PA)C(PA).$$
Since $(PA)C=(AC)^2S-AC$ is a Riesz operator by assumption, then it follows from Proposition \ref{prop3.1} that $B(PA)$ is a Riesz operator. Therefore $BA$ is generalized Drazin-Riesz invertible and $BS^2A$ is a generalized Drazin-inverse of $BA$.
\smallskip

In similar way, we prove the opposite implication.
\end{proof}

\begin{rem} If $A$ and $B\in\mathcal{L}(X)$ such that $ABA=A^2$ and $BAB=B^2$, then 
\begin{equation}\label{deq1}A(BA)^2=ABAIA=AIABA=(AI)^2A
\end{equation} and 
\begin{equation}\label{deq2} B(AB)^2=BABIB=BIBAB=(BI)^2B.
\end{equation}
Then it follows from (\ref{deq1}) and (\ref{deq2} that $A$, $B$, $BA$ and $AB$ share above spectral properties. So we retrieve the results of \cite{Du}.
\end{rem}

In the following two examples, the common spectral properties for $AC$ and $BA$ can only followed directly from the above results, but not from the corresponding ones in \cite{Sch1, Sch2, Du, ZZ1, TK}.
\begin{ex}\label{Ex1} Let $P$ be a non trivial idempotent on $X$. Let $A$, $B$ and $C$ defined on $X\oplus X\oplus X$ by
\[ 
A = \left(\begin{array}{ccc} 0 & I & 0\\
                                0 & P & 0\\
                                0 & 0 & 0
                                                  \end{array}  \right)\mbox{,  }
B = \left(\begin{array}{ccc} I & 0 & 0\\
                                0 & I & 0\\
                                0 & 0 & 0
                                                  \end{array}  \right)  \mbox{ and }
 C = \left(\begin{array}{ccc} 0 & 0 & 0\\
                                I & 0 & 0\\
                                0 & I & 0
                                                  \end{array}  \right).
\]                                
Then $A(BA)^2=ABACA=ACABA=(AC)^2A$, while  $ABA \neq ACA$ and $BAB\neq B^2$.
\end{ex}

\begin{ex} Let $A$ and $B$ be as in Example \ref{Ex1}  and let $C$ be defined on $X\oplus X\oplus X$ by
\[ 
 C = \left(\begin{array}{ccc} 0 & 0 & 0\\
                               P & 0 & 0\\
                                0 & I & 0
                                                  \end{array}  \right).
\]                                
Then $A(BA)^2=ABACA=ACABA=(AC)^2A$, while  $ABA \neq ACA$ and $BAB\neq B^2$. 
\end{ex}


\end{document}